\documentclass{article}
\usepackage[utf8]{inputenc}
\usepackage{graphicx}
\usepackage{mathtools,amsmath,amssymb,amsfonts,amsthm}
\usepackage{mathalfa}
\usepackage{tikz}
\usetikzlibrary{calc}
\usepackage{color}
\usepackage{url}

\usepackage{enumitem}

\usepackage{subcaption}

\graphicspath{{./figures/}}

\usepackage[english]{babel}

\usepackage{tabularx}

\usepackage{xcolor}

\usepackage{verbatim}
\usepackage{xspace}

\usepackage{amsfonts}
\usepackage{mathtools}
\usepackage{amsmath}

\usepackage{amssymb}

\usepackage{setspace}

\usepackage{hhline}

\usepackage{fullpage}
\usepackage{ifthen}
\usepackage{tikz}
\usetikzlibrary{automata,positioning, arrows}

\theoremstyle{plain}
\newtheorem{thm}{Theorem}
\newtheorem{lemma}{Lemma}
\newtheorem{prop}{Proposition}

\newtheorem{proper}{Properties}
\theoremstyle{definition}

\newtheorem{defn}{Definition}

\theoremstyle{remark}
\newtheorem{myexample}{Example}
\newtheorem{rmk}{Remark}

\newenvironment{example}[1][0]
{ 
  \ifthenelse{\equal{#1}{0}}{
  \myexample
}
{ 
  \renewcommand\themyexample{#1}\myexample
  \addtocounter{myexample}{-1}
}
}
{\endmyexample}

\newcommand{\R}{\mathbb{R}}{}
\newcommand{\N}{\mathbb{N}}

\newcommand{\cA}{\mathcal{A}}

\newcommand{\cC}{\mathcal{C}}
\newcommand{\cD}{\mathcal{D}}
\newcommand{\cH}{\mathcal{H}}

\newcommand{\wi}{\hat {\imath}}

\newcommand{\cK}{\mathcal{K}}
\newcommand{\cKL}{\mathcal{KL}}
\newcommand{\cM}{\mathcal{M}}

\newcommand{\cG}{\mathcal{G}}

\newcommand{\cS}{\mathcal{S}}
\newcommand{\cQ}{\mathcal{Q}}

\newcommand{\cP}{\mathcal{P}}

\newcommand{\cO}{\mathcal{O}}

\DeclarePairedDelimiterX{\inp}[2]{\langle}{\rangle}{#1, #2}

\title{Interpretability of Path-Complete Techniques and Memory-based Lyapunov functions}
\author{Matteo Della Rossa \and Rapha\"el M. Jungers
\thanks{
RJ is a FNRS honorary Research Associate. This project has received funding from the European Research Council (ERC) under the European Union's Horizon 2020 research and innovation program under grant agreement No 864017 - L2C. RJ is also supported by the Innoviris Foundation and the FNRS (Chist-Era Druid-net). }
\thanks{M. Della Rossa and R. Jungers are with ICTEAM,
        UCLouvain, Louvain-la-Neuve, Belgium).
         {\tt\small matteo.dellarossa@uclouvain.be}}%
}

\begin{document}

\maketitle
\begin{abstract}
We study path-complete Lyapunov functions, which are stability criteria for switched systems, described by a combinatorial component (namely, an automaton), and a functional component (a set of candidate Lyapunov functions, called \emph{the template}).
We introduce a class of criteria based on what we call memory-based Lyapunov functions, which generalize several techniques in the literature. Our main result is an equivalence result: any path-complete Lyapunov function is equivalent to a memory-based Lyapunov function, however defined on another template.
We show the usefulness of our result in terms of numerical efficiency via an academic example.
\end{abstract}
\section{Introduction}
Modern engineering systems are increasingly complex, and require  safety/stability guarantees, due to their interaction  with humans in the loop as, for example, smart automotive vehicles, smart grids, or articulated robots.  However, typically, these recent advances, such as for instance artificial neural networks, or random forests, are essentially ``black box'' algorithms: they are generated by massive iterative optimization schemes, processing a huge amount of data harvested automatically by sensors or other computerized sources. These tools rarely come with guarantees in terms of safety or performance, which hampers their implementation in real-world applications. 

In view of this, control theory faces a critical need for interpretable AI, where the output of massive computations can be analyzed and understood by engineers (or by a smart algorithm) in order to translate the produced knowledge into an educated analysis, which in turn can lead to a guarantee in terms of safety, efficiency, or any other objective.
A paradigmatic example of powerful blackbox data processing algorithms are LMIs; in control theory they remain a central tool for engineers. If the control task is not too complex, it is often possible to rewrite the control problem as a set of LMIs, which can be solved efficiently thanks to interior point methods.  Even more, in more involved situations, when the control task is more intricate, or for more complicated dynamical systems, relaxation techniques (Sum-Of-Squares, the S-procedure, convex relaxation,...) often allow to write LMI formulations of the control problem, potentially at the price of adding conservatism. Typically, the obtained LMIs may become a large set of algebraic inequalities, from which one can hardly extract physical sense. Being able to \emph{interpret} (that is, provide a physical meaning) to these algebraic criteria would enable engineers to verify the obtained solutions in practice, or generalize them to the real-world problem, beyond the used simplified model.

In this paper, we tackle interpretability for a well-established LMI-based stability analysis technique for hybrid systems, namely path-complete stability analysis (see Section~\ref{sec:Prelim} for definitions). We reinterpret this framework as a \emph{compressive memory} stability criterion: by making use of tools from automata and language theory, we show that any Path-complete stability criterion is in fact a Lyapunov function which takes as argument not only a point in the state space, but also a ``memory observation'', which can be of variable length. Our result is of theoretical interest, because of the interpretability mentioned above, but it is also of practical interest. Indeed, we show in Section~\ref{sec:NumExample} that our result brings more than just a qualitative interpretation: it can also provide a smart way to improve the numerical efficiency of stability analysis criteria. 
This reinterpretation of path-complete criteria in terms of memory  also allows to possibly apply these techniques not only for switched systems but for more general hybrid systems with output.

Our main result bears similarity with several classical and more recent works in the context of hybrid systems control, where memory (or similar concepts) is used as a proxy for refining an abstract representation of a dynamical system. In~\cite{SchmTab15,MajuOzat2020} and references therein, a transition system representing the possible memory-states of an observer is built in order to provide an abstraction of a given system. A novelty of our work is that we exploit memory-based transition systems, not for the sake of obtaining an abstraction of a given system, but in order to generate, or analyze, optimization techniques aiming at solving a particular problem (here, stability). Closer to our setting, in the context of Lyapunov techniques for switched systems, concepts of memory or the use of information concerning the past switching sequence, were proposed for example in~\cite{LeeDull06,Thesis:Essick,PeiLax21}. However, these works only consider fixed memory length, while our approach, relying on automata and languages, allows to handle adaptable-length memory, which can dramatically improve scalability. 

In this manuscript, after having recalled the setting and definitions in Section~\ref{sec:Prelim}, we present our equivalence results in Section~\ref{Sec:MainSec}. In Section~\ref{sec:NumExample} we apply our result to a simple academic example before concluding in Section~\ref{sec:Conclu} with some final remarks.

\textbf{Notation:} A function $\alpha:\R_{\geq 0}\to \R_{\geq 0}$ is of \emph{class $\cK$} ($\alpha \in \cK$) if it is continuous, $\alpha(0)=0$, and strictly increasing; it is of \emph{class $\cK_\infty$} if, in addition, it is unbounded. A continuous function $\beta:\R_+\times \R_+\to \R_+$ is of \emph{class $\mathcal{KL}$} if $\beta(\cdot,s)$ is of class $\cK$ for all $s$, and $\beta(r,\cdot)$ is decreasing and $\beta(r,s)\to 0$ as $s\to\infty$, for all $r$. Given $n,m\in \N$, $\cC(\R^n,\R^m)$ denotes the set of continuous functions between $\R^n$ and $\R^m$.
\section{Preliminaries}\label{sec:Prelim}
Given a discrete set of symbols $\cS$ (called, from now on, the \emph{alphabet}) and a set  $F=\{f_i\}_{i\in \cS}\subset \cC(\R^n,\R^n)$ we consider a discrete-time dynamical system of the form
\begin{equation}\label{eq:DynamicSystem}
\begin{aligned}
x(k+1)=f_{\sigma(k)}(x(k)),\;\;\;x(0)=x_0,
\end{aligned}
\end{equation}
where $\sigma:\N\to \cS$ is the so-called \emph{switching signal}. 
\begin{defn}\label{defn:Stability}
Given $\{f_i\}_{i\in \cS}\subset \cC(\R^n,\R^n)$, system~\eqref{eq:DynamicSystem} is said to be~\emph{uniformly globally asymptotically stable} (UGAS) if there exists a $\beta\in \cKL$ such that
\[
 \forall \sigma:\N\to \cS, \,\forall x_0\in \R^n, \,\forall k\in \N,\;\;|\Psi_\sigma(k,x_0)|\leq \beta(k,|x_0|),
\]
where $\Psi_\sigma(k,x_0)$ denotes the solution of~\eqref{eq:DynamicSystem} with respect to the signal $\sigma$, starting at $x_0$ and evaluated at time $k\in \N$.
\end{defn}
Several Lyapunov methods have been proposed in order to study asymptotic stability of system~\eqref{eq:DynamicSystem}, for an overview see~\cite{Lib03} and references therein. 
In what follows we recall the formal definition of the \emph{path-complete Lyapunov framework}, which can be seen as a \emph{combinatorial} multiple Lyapunov functions approach that makes use of directed and labeled graphs to encode the structure of the conditions.

 Given an alphabet $\cS$, a \emph{graph} $\cG$ on $\cS$ is defined by $\cG=(N,E)$ where $N$ and $E \subseteq N \times N \times \cS$ are the set of nodes and  of labeled edges, respectively. Any Lyapunov certificate for~\eqref{eq:DynamicSystem} needs to ensure asymptotic stability \emph{uniformly} over the set of any possible $\sigma:\N\to \cS$. In the graph framework, this requirement is formalized by a combinatorial property, the \emph{path-completeness} introduced in \cite{AJPR:14} and recalled here. 

\begin{defn}[Path-complete graph]
Given  an alphabet $\cS$, a graph $\cG=(N,E)$  on $\cS$ is \emph{path-complete} if, for any $K \geq 1$ and any sequence $\wi = (j_1\,\dots\,j_K)\in \cS^K$, there exists a \emph{path}  $\{(a_k,a_{k+1},j_k)\}_{k=1, \dots,K}$ such that $(a_k,a_{k+1},j_k)\in E$, for each $1\leq k\leq K$.
\end{defn}
Before recalling the application of path-complete graphs to stability analysis of~\eqref{eq:DynamicSystem}, we introduce some additional notions of graph theory which will be used in what follows.
\begin{proper}\label{proper:Graphs}
A graph $\cG=(N,E)$ on $\cS$ is said to be:
\begin{itemize}[leftmargin=*]
\item \emph{complete}, if, for any $a\in N$ and any $i\in \cS$, there exists $b\in N$ such that $(a,b,i)\in E$. 
\item \emph{deterministic}, if,  for any $a\in N$ and any $i\in \cS$, there exists \emph{at most} one $b\in N$ such that $(a,b,i)\in E$.
  \end{itemize}
  A complete graph is in particular path-complete. Given $\cG=(N,E)$, its \emph{dual graph} $\cG^\top$ is defined by $\cG^\top=(N^\top,E^\top)$ with $N^\top=N$ and $(a,b,i)\in E^\top$ $\Leftrightarrow$ $(b,a,i)\in E$.  A graph $\cG$ is path-complete if and only if so is $\cG^\top$.
\end{proper}

Concluding this section, we recall from~\cite{AJPR:14,PEDJ:16,PhiAth19} the concept of path-complete Lyapunov functions and the corresponding stability result.

\begin{defn}[Path-complete Lyapunov Function]\label{defn:PCLF}
Given a switching system $F = \{ f_i\}_{i\in \cS}\subset \cC(\R^n, \R^n)$, a \emph{path-complete Lyapunov function} (PCLF) for $F$ is given by $(V, \cG)$ with $\cG = (N,E)$ a path-complete graph on $\cS$, and $V = \{ V_s \mid s \in N \} \subseteq \cC(\R^n ,\R)$  such that, for some $\alpha_1,\alpha_2\in \cK_{\infty}$ and some  $\gamma\in [0,1)$, the following inequalities are satisfied: 
\begin{subequations}
    \begin{align}
\alpha_1(|x|)\leq V_s(x)\leq \alpha_2(|x|),\;&\forall s\in N,\;\forall x\in \R^n;\label{eq:Sandwich0}\\
V_b(f_i(x))\leq\gamma V_a(x),\;&\forall \, (a,b,i) \in E, \; \forall x \in \mathbb{R}^n.
\label{eq:LyapunovInequalitiesStorng}
  \end{align}
\end{subequations}
\end{defn}
\begin{prop}[Stability result,~\cite{AJPR:14,PhiAth19}]
Given any $F = \{ f_i\;\;\vert\;i\in \cS\}\subset \cC(\R^n, \R^n)$; if there exists a PCLF for $F$, then system~\eqref{eq:DynamicSystem} is~UGAS.
\end{prop}

\begin{rmk}
Here and in what follows, we focus on the UGAS concept only, and the corresponding path-complete Lyapunov characterization. The same ideas and techniques can be adapted, mutatis mutandis, for the stability (without convergence), and or local/practical notions of convergence. These extensions/adaptations are not explicitly developed here, for the sake of clarity. Indeed, this work does not focus on stability of particular dynamical systems, but rather we provide a meta-analysis of the stability criteria themselves.
\end{rmk}

\section{Memory-based Lyapunov functions and Equivalence with PCLF}\label{Sec:MainSec}
In this section we introduce \emph{memory-based} Lyapunov functions and we provide our main result: the equivalence between path-complete stability techniques and memory-based criteria.

\subsection{Memory-based Lyapunov functions}
In this subsection, we give a language-based interpretation of Lyapunov stability criteria for~\eqref{eq:DynamicSystem}: we will interpret the (admissible) past switching events as sequences/strings (rather than functions/signals). 
 We thus introduce the following formal definition.
\begin{defn}
Given an alphabet $\cS$, by $\cS^{\star}$ we denote the \emph{language generated by $\cS$}, defined by
\[
\cS^{\star}:=\bigcup_{k\in \N}\cS^k,
\] 
i.e. the set of all finite strings\footnote{$\cS^\star$ is also called the \emph{Kleene-closure} of $\cS$.} of elements of $\cS$.  We adopt the convention that $\cS^0=\{\epsilon\}$, with $\epsilon$ an additional symbol, denoting the \emph{empty string}.
\end{defn}
With an abuse of notation, when needed, we interpret elements of $\cS^\star$ as \emph{partial functions} $\sigma:\N\to \cS$, with the convention that, if $\sigma\in \cS^K$ (for some $K\in \N$), we write $\sigma(k)=\bot$, for $k>K$.

Given $\wi=(i_1,\dots,i_K)\in \cS^K$ and a symbol $h\in \cS$, we denote by $h\wi\in \Sigma^{K+1}$ the string defined by $(h,i_1,\dots, i_K)$.

We provide the following definition, which will allow us to consider finite partitions of $\cS^\star$.
\begin{defn}
A subset $B\subset \cS^\star$ is said to be a \emph{regular language}  if it is recognized by a finite-state \emph{automaton}.
\end{defn}
We do not give the formal definition of finite state automata, for which we refer to~\cite[Section 2.4]{CassandrasLafortune08}, where 
more insight into regular languages can also be found. Intuitively, a finite state automaton is a graph $\cG=(N,E)$ on an alphabet $\cS$, with some states in $N$ marked as \emph{initial} nodes and some marked as \emph{accepting} nodes, see Figure~1 for an example. 

Given a regular language $B$ and any symbol $h\in \cS$, $hB$ denotes the regular language $hB:=\{h\wi\;\;\vert\;\wi\in B\}$.
\begin{example}\label{example:prefixclasss1}
We provide a particular case of regular languages.
A \emph{prefix class} \emph{of length $K\in \N$} of $\cS^\star$  is a regular language defined by a multindex $\wi=(i_1,\dots, i_K)\in \cS^K$, by
\[
[\wi]:=\left \{\sigma\in \cS^\star\;\;\vert\; \sigma(k)\in\{i_k, \bot\}\,\; \forall k\in [1,K]\right\}.
\]
It is clear that any prefix class $[\wi]$ is a regular language, since it is recognized by the finite state automaton depicted in Figure~1. Prefix classes are particularly convenient in our context, since, given a prefix class of length $K$ denoted by $\wi=[i_1,\dots, i_K]$, its concatenation $h\wi$ is the prefix class of length $K+1$ given by $[h,i_1,\dots, i_K]$.
\end{example}
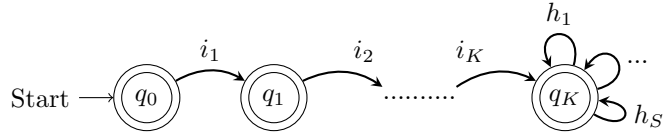
\begin{figure}[t!]
  \color{black}
  \centering
  \tikzstyle{accepting}=[path picture={%
  \draw let 
    \p1 = (path picture bounding box.east),
    \p2 = (path picture bounding box.center)
    in
      (\p2) circle (\x1 - \x2 - 3pt);
  }]
\begin{tikzpicture} [
    node distance = 1.69cm, 
    on grid, 
    auto,
    every loop/.style={stealth-}]
 
\node (q0) [state, 
    initial, 
   , accepting,
    initial text = {Start}] {$q_0$};
 
\node (q1) [state,
    right = of q0, accepting] {$q_1$};

    \node(qfict)[right =of q1]{$.....$};
 \node(qfict1)[right =of qfict,xshift=-1.2cm]{$.....$};

     \node(qK)[state,
    right = of qfict1, accepting]{$q_{K}$};
 
\path [-stealth, thick]
    (q0) edge[bend left] node {$i_1$}   (q1)
    (q1) edge[bend left] node {$i_2$}   (qfict)
    (qfict1) edge[bend left] node {$i_K$}   (qK)
    (qK) edge [loop above, out=75, in=115, looseness=5]  node {$h_1$}(qK)
    (qK) edge [loop right, out=15, in=55, looseness=5]  node {$...$}(qK)
     (qK) edge [loop right, out=-30, in=-5, looseness=7]  node {$h_S$}(qK);
    \end{tikzpicture}
\label{figute:Automataprefix class}
\caption{Automaton recognizing a prefix class $\wi=[i_1,\dots,i_K]\subset\cS^\star$, with $\cS=\{h_1,\dots, h_S\}$. The ``Start'' symbol denotes the starting nodes, while the double circle denotes an accepting node.}
\end{figure}

We now introduce a particular class of partitions of $\cS^\star$.
\begin{defn}[Covering family of languages]\label{defn:Covering}
Given a finite set of regular languages $\cC=\{B_1,\dots, B_M\}$, with $B_i\subset \cS^\star$ we say that $\cC$ is a  \emph{covering family of languages} if 
\begin{subequations}
\begin{align}
&\bigcup_{j=1}^M B_j=\cS^\star;\label{eq:covering}\\
\forall B\in\cC,\;\;\forall h\in \cS,\;\;&\exists \,C\in \cC\text{ such that }\;hB\subseteq C.\label{eq:Consistency}
\end{align}
\end{subequations}
\end{defn}
Intuitively, condition~\eqref{eq:covering} formalizes the fact that the family $\cC$ covers the whole set of possible strings, while condition~\eqref{eq:Consistency} specifies that, given $B\in \cC$, for any symbol $h\in \cS$, there is (at least) a subset of the family containing all the corresponding concatenations. 
\begin{defn}[Memory-based LF]\label{defn:FiniteMBLF}
Given a covering family of languages $\cC=\left\{B_1,\dots,B_N\right\}$, suppose there exist a continuous function $W:\cC\times \R^n\to \R$, functions $\alpha_1,\alpha_2\in \cK_\infty$, scalar $\gamma\in [0,1)$ such that
\begin{subequations}
\begin{align}\label{eq:Sandwich1}
&\alpha_1(|x|)\leq W(B,x)\leq \alpha_2(|x|),\;\;\; \forall\,(B,x)\in \cC\times \R^n;\\
&W(C,f_h(x))\leq \gamma W(B,x),\,\forall B\in \cC,\forall h\in \cS, \forall C\in \cC,\,\text{ s.t. } hB\subseteq C, \forall x\in \R^n. \label{eq:DecreasingFiniteMemory}
\end{align}
\end{subequations}
Then $W$ is said to be a \emph{$\cC$-memory-based Lyapunov function} ($\cC$-MBLF) for system~\eqref{eq:DynamicSystem}. 
\end{defn}

Any memory-based Lyapunov function provides a sufficient criteria for UGAS, as proven in the following statement.
\begin{lemma}\label{lemma:LyapFunction}
Given a covering family of language $\cC$, if there exists a $\cC$-memory-based Lyapunov function $W:\cC\times \R^n\to \R$ for~\eqref{eq:DynamicSystem}, then system~\eqref{eq:DynamicSystem} is UGAS.\end{lemma}
\begin{proof}
Consider any $x_0\in \R^n$, any $\sigma:\N\to \cS$ and any $k\in \N$. From conditions~\eqref{eq:Sandwich1},~\eqref{eq:DecreasingFiniteMemory}, we have
\[
\begin{aligned}
\alpha_1(|\Psi_{\sigma}(k,x_0)|)&\leq W(B_k, f_{\sigma(k-1)}(\Psi_{\sigma}(k-1,x_0)\,)\leq \dots \\&\leq \gamma^k W(B_0,x_0)\leq \gamma^k\alpha_2(|x_0|),
\end{aligned}
\]
where $B_0,\dots B_k\in \cC$ are chosen such that $\epsilon\in B_0$, $\sigma(0)\in B_1$, ... , $(\sigma(k-1),\dots, \sigma(0))\in B_k$ and such that $\sigma(h)B_{h-1}\subset B_{h}$ for any $h\in \{1,\dots k\}$, which is possible by properties of $\cC$ in Definition~\ref{defn:Covering}.
We thus have $|\Psi_{\sigma}(k,x_0)|\leq \beta(|x_0|,k)$, where we defined $\beta(s,k):=\alpha_1^{-1}(\gamma^k\alpha_2(s))$.
It is easy to see that $\beta\in \cKL$, and it does \emph{not} depend on $\sigma:\N\to \cS$ nor $x_0\in \R^n$, thus proving UGAS of~\eqref{eq:DynamicSystem}.
\end{proof}
\subsection{From MBLF to Path-complete Lyapunov functions}
We prove here that  memory-based Lyapunov functions can be re-interpreted as path-complete Lyapunov functions.
\begin{thm}\label{thm:FirstTransl}
Given an alphabet $\cS$, consider any covering family of languages $\cC$. There exist a complete graph $\cG_\cC=(N_\cC,E_\cC)$ and a 1-to-1 map $\Phi: N_\cC\to \cC$ such that, for any system $F=\{f_i\}_{i\in \cS}\subset \cC(\R^n,\R^n)$, the following holds:

  $W:\cC\times\R^n\to \R$ is a $\cC$-memory-based Lyapunov function for $F$ if and only if $(\{W({\Phi(s),\cdot)}\}_{s\in N_\cC}, \cG_\cC)$ is a  path-complete Lyapunov function  for $F$.
\end{thm}
\begin{proof}
Let us consider a covering family of regular languages $\cC$.
We define a graph $\cG_\cC=(N_\cC,E_\cC)$ with $|N_\cC|=|\cC|$ nodes in a 1-to-1 correspondence to the languages in $\cC$: for any $B\in \cC$ we denote by $s_B\in N_\cC$ the corresponding node in $\cG_\cG$, i.e. $\Phi(s_B)=B$. We then define the edge set $E_\cG$ by the following condition:
\[
(s_B,s_C,h)\in E_\cC\;\; \Leftrightarrow\;\; hB\subseteq C.
\]
By the properties in Definition~\ref{defn:Covering}, the graph is complete. Then, we define $\{V_s\}_{s\in N_\cC}\subset \cC(\R^n,\R^n)$ by $V_{s_C}(x):=W(C,x)$.  Given any $F=\{f_i\}_{i\in \cS}$, conditions in Definition~\ref{defn:PCLF} are satisfied if and only if conditions in Definition~\ref{defn:FiniteMBLF} are.
\end{proof}

\begin{example}\label{example:prefixclasss2}
Consider $\cS=\{a,b\}$, a particular covering family of regular languages is given by $\cD:=\{[aa],[ab],[b]\}$, i.e. a covering of prefix classes, as introduced in Example~\ref{example:prefixclasss1}. The corresponding complete graph given by Theorem~\ref{thm:FirstTransl}, denoted by $\cH$, is depicted in Figure~\ref{figure:Figure1}. Since, in this case, for any prefix class $B$ in $\cD$ and any $h\in \cS$ there exists a \emph{unique} prefix class $C\in \cD$ such that $hB\subset C$, the corresponding graph is also deterministic. 
\begin{figure}[t!]
  \color{black}
  \centering
  \begin{tikzpicture}%
  [>=stealth,
  shorten >=0.8pt,
  node distance=0.5cm,
  on grid,
  auto,
  every state/.style={draw=black, fill=white, thick}
  ]
  \node[state] (left)  [yshift=2cm]                {$[b]$};
  \node[state] (right) [right=of left, xshift=3cm] {$[ab]$};
  \node[state] (below) [below right=of left, xshift=1.5cm, yshift=-2cm]{$[aa]$};
  \path[->]
  (left) edge[loop left=60]     node                      {$b$} (left)
  (left) edge[bend left=15]     node                      {$a$} (right)
  (right) edge[bend left=15]     node                      {$b$} (left)
  (below) edge[bend left=10]     node                     {$b$} (left)
  (right) edge[bend left=10]     node                     {$a$} (below)
  (below) edge [loop below=10] node[below] {$a$} (below)
;
\end{tikzpicture}
\caption{Graph $\cH$, path-complete representation of memory-based Lyapunov conditions arising considering  $\cD=\{[aa],[ab],[b]\}$, as in Example~\ref{example:prefixclasss2}. }
\label{figure:Figure1}
\end{figure}
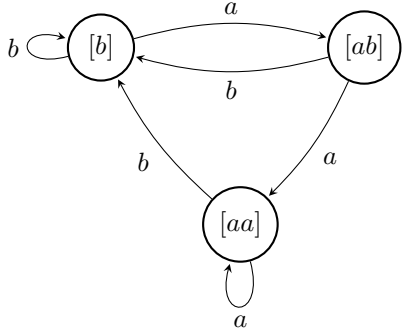

\end{example}

\begin{rmk}[Prefix classes conditions in the literature]\label{rem:Cyliders}
The relations between covering of prefix classes (as defined in Example~\ref{example:prefixclasss1}) of \emph{the same length}, stability conditions for switched systems based on ``memory'' and graph theory was partially illustrated in~\cite{LeeDull06,Thesis:Essick,DonDull20}. In these works, the Authors use a covering made by the $|\cS|^K$ prefix classes of length $K$, thereby formalizing the idea of storing, at each instant, information on the previous $K$ values of the switching signal. From a graph-theory point of view, these conditions based on fixed length prefix classes give rise to the De-Bruijn graph structure. This family of directed graphs was introduced in the seminal paper~\cite{DeBru46}, in a language theory-context, to formalize the idea of words/strings ``with the same prefixes/past''; it is thus not surprising that it arises in this context. One of our main contribution is that we generalize the results based on fixed-length prefix classes memory (which are now re-obtained as corollary) obtaining a general equivalence between graph-based and memory-based Lyapunov conditions. 
\end{rmk}
\subsection{From Path-Complete Lyapunov functions to MBLF}
In this subsection we show that any path-complete Lyapunov structure leads to a memory-based Lyapunov function.

We present our main result first, developing its proof in the remainder of this subsection.
\begin{thm}\label{thm:MainThm}
 For any path-complete graph $\cG=(N,E)$ there exists a covering family $\cC_\cG$ such that the following holds:

Given any $F=\{f_i\}_{i\in \cS}\subset \cC(\R^n,\R^n)$, if there exists a function $V = \{ V_s \mid s \in N \} \subseteq \cC(\R^n ,\R)$ such that  $(V,\cG)$ is a PCLF for $F$ then there exists a $\cC_\cG$-memory-based Lyapunov function $W:\cC_\cG\times \R^n\to \R^n$ such that  $W(B,\,\cdot)\in \cM(V):=\left\{\max_{s\in P}V_s(\,\cdot\,)\;\;\vert\;\;P\subseteq N\right \}$ for all $B\in \cC$.
\end{thm}
The proof of this Theorem is broken into technical lemmas, for the sake of readability.
\begin{lemma}[Adapted from~\cite{PhiAth19}]\label{lemma:Philippe}
Given any path-complete graph $\cG=(N,E)$ on $\cS$, there exists a graph $\cO_\cG=(N_\cO,E_\cO)$, called the \emph{observer graph}, with $N_\cO\subset \cP(N)$ (i.e. nodes of $\cO_\cG$ correspond to subsets of nodes of $\cG$) which is deterministic and complete, such that the following property holds: 

Consider any system $F=\{f_i\}_{i\in \cS}$. If $V = \{ V_s \mid s \in N \} \subseteq \cC(\R^n ,\R)$ is such that $(V,\cG)$ is a PCLF for $F$ then, defining $W=\{\max_{s\in P}{V_s}\;\vert\; P\in N_\cO\}$ one has that $(W,\cO_\cG)$ is a PCLF for $F$.
\end{lemma}
\begin{proof}
The construction of the observer graph is sketched in what follows. We point out that it is recalled here from~\cite[Theorem III.8]{PhiAth19} for the sake of self-containment and in a version more suited for our purposes. Similar construction in a more general context can be found in ~\cite[Section 2.3.4.]{CassandrasLafortune08}.

Given $\cG=(N,E)$, the graph $\cO_\cG=(N_\cO,E_\cO)$ is constructed by steps. We recursively define the~\emph{observer graph} $\cO_\cG=(N_\cO,E_\cO)$ as follows: first, we consider $N_\cO:=\{N\}$, $E_\cO=\emptyset$ and for any $h\in \cS$, defining the set 
\[
N(h):=\{q\in N\;\vert\;\exists p\in N\text{ s.t } (p,q,h)\in E\}\subseteq N,
\]
we add to $N_\cO$ and $E_\cO$, respectively, the nodes and edges
\[
N(h)\in N_\cO,\;\;\;\;(N,N(h),h)\in E_\cO.
\]
Then we iterate the procedure,  considering strings of length $K$: for any $\wi=(h_1,h_2,\dots, h_K)\in \cS^K$, we introduce the notation $\wi^-:=(h_2,\dots,h_K)$ and we define
\[
N(\wi)=\left\{q\in N\;\Big \vert\;\begin{aligned} &\exists p\in N(h_2,h_3,\dots,h_{K})\\&\text{ s.t } (p,q,h_1)\in E \end{aligned} \right\}\subseteq N,
\]
and we add the nodes and edges
\begin{equation}\label{eq:ObserverIDea}
N(\wi)\in N_\cO,\;\;\;\;(N(\wi^-), N(\wi),h)\in E_\cO.
\end{equation}
We underline that the same subset $Q\subset P$ can correspond to several (and in general infinite) strings $\wi\in \cS^K$.
By path-completeness of $\cG$, this procedure will never reach the empty set $\emptyset\in \cP(N)$: if, by contradiction, for some $\wi\in \cS^\star$ we have $N(\wi)=\emptyset$, then there exists no path in $\cG$ (starting from any node) labeled by $\wi$, this contradicting Definition~\ref{defn:PCLF}. 

 By finiteness of $\cP(N)$ (the power set of $N$), this procedure ends after a finite number of set. For the explicit implementation of the algorithm, we refer to~\cite{PhiAth19}. By construction, the obtained graph is complete and deterministic.

The second statement in Lemma~\ref{lemma:Philippe}, leading to the construction of a path-complete function $(W,\cO_\cG)$ composed by pointwise maxima of a path-complete function $(V,\cG)$ is proven in~\cite[Proposition III.1]{PhiAth19}, to which we refer.
\end{proof}
It can be shown that there exists a unique strongly connected subgraph of $\cO_\cG$ (and thus, also complete and deterministic), as proven in~\cite{PhiAth19}. This graph will provide somehow a more concise memory-based representation of the considered path-complete Lyapunov function, but, for the sake of clarity, we do not develop the argument in this submission.

We now show that given a path-complete graph $\cG$, its observer $\cO_\cG$ allows to define a particular covering family of languages, which will be a crucial step for the proof of Theorem~\ref{thm:MainThm}.

\begin{lemma}\label{lemma:FirstReduction}
Consider any path-complete graph $\cG=(N,E)$, and its observer graph $\cO_\cG=(N_\cO,E_\cO)$ defined in proof of Lemma~\ref{lemma:Philippe}. To any $P\in N_\cO$ we can associate a regular language $B_P\subset\cS^\star$ such that the corresponding $\cC_\cG=\{B_P\}_{P\in N_\cO}$, is a covering family of languages, with the property that, for any $P,Q\in N_\cO$, we have
\begin{equation}\label{eq:ObserevrisCool}
(P,Q,h)\in E_\cO\;\;\Leftrightarrow\;\;hB_P\subseteq B_Q.
\end{equation}
\end{lemma}
\begin{proof}
Given $\cG=(N,E)$, we consider its observer graph $\cO_\cG=(N_\cO,E_\cO)$ as defined in proof of Lemma~\ref{lemma:Philippe}.
 For any $P\in N_\cO$, we define the corresponding regular language, $B_P\subset \cS^\star$ by defining an automata $\cA_P$ as follows:
 \begin{itemize}[leftmargin=*]
\item Reverse the graph $\cO_\cG$, i.e., consider the dual graph $\cO_\cG^\top$, (recall the definition in Properties~\ref{proper:Graphs}).
\item Mark the node $P$ as the unique initial state.
\item Mark the node  $N$ as the unique accepting state.
 \end{itemize}
 Then, $B_P$ is the set of strings recognized by $\cA_P$.

We then prove that $\cC_\cG:=\{B_P\}_{P\in N_\cO}$ is a covering family of languages. First, observe that the definition of the observer graph in proof of Lemma~\ref{lemma:Philippe} implies condition~\eqref{eq:covering}. Indeed, in this definition, we consider all the possible strings $\wi\in \cS^\star$. Also, condition~\eqref{eq:Consistency} is implied by completeness of $\cO_\cG$ and by condition~\eqref{eq:ObserevrisCool} which is proven in what follows:

Suppose that $(P,Q,h)\in E_\cO$, for some $P,Q\in N_\cO$ and $h\in \cS$. Consider any string $\wi\in B_P$, i.e. there exists a path $\mathbf{p}$ in $\cO_\cG^\top$ starting at $P$, ending at $N$ and labeled by $\wi$. Then, there also exists a path in $\cO_\cG^\top$ starting at $Q$ labeled by $h\wi$ and ending at $N$: simply concatenating the edge $(Q,P,h)\in E^\top_\cO$ to the previously chosen path $\mathbf{p}$. Thus we have proven that $h\wi\in B_Q$.
For the other direction suppose $hB_P\subseteq B_Q$, then consider a string $\wi\in B_P$; by definition of $B_P$ there is a path in $\cO_\cG^\top$ from the node $P$ to the node $N$ labeled by $\wi$. The fact that $h\wi\in B_Q$ implies that there is a path in $\cO_\cG^\top$ from the node $Q$ to the node $N$  labeled by $h\wi$. This, by the construction provided in proof of Lemma~\ref{lemma:Philippe}, implies that $(P,Q,\wi)\in E_\cO$ concluding the proof.
\end{proof}

 We can now prove our ``translation'' result.
\begin{proof}[Proof of Theorem~\ref{thm:MainThm}]
Given any $\cG=(N,E)$ we build the corresponding observer graph $\cO_\cG=(N_\cO,E_\cO)$, as presented in the proof of Lemma~\ref{lemma:Philippe}. We then consider the corresponding covering family of languages $\cC_\cG=\{B_P\}_{P\in N_\cO}$ as in Lemma~\ref{lemma:FirstReduction}. 
Consider any system $F=\{f_i\}_{i\in \cS}$ and suppose that $V=\{V_s\;\vert\;s\in P\}$ is such that $(\cG,V)$ is a PCLF for $F$.
We define $W:\cC_\cG\times \R^n\to \R$, by
\[
W(B_P,x):=\max_{s\in P} \{V_s(x)\},
\]
  By the second statement in Lemma~\ref{lemma:Philippe} and by equation~\eqref{eq:ObserevrisCool} the function $W$ satisfies the conditions in Definition~\ref{defn:FiniteMBLF}, concluding the proof.
\end{proof}
Concluding this section, we discuss the possibility of having ``dual'' results, involving the concept of \emph{``future''} based Lyapunov functions.
\begin{rmk}[Memory vs Future Characterization]
In this work, we focus on memory-based Lyapunov functions (as defined in Definition~\ref{defn:FiniteMBLF}), and we prove that this formalism is equivalent to the path-complete Lyapunov functions framework (cfr. Definition~\ref{defn:PCLF}), in Theorem~\ref{thm:FirstTransl} and Theorem~\ref{thm:MainThm}.
To prove this equivalence, it is necessary, given a path-complete Lyapunov function, to consider pointwise maxima of the original composing functions. By duality, i.e. developing the arguments for the dual graphs (and considering pointwise minima), similar results can be obtained for ``future''-based Lyapunov functions; this will be the topic of future research.
\end{rmk}

\section{Numerical Example}\label{sec:NumExample}
In this section, we present a numerical example, showing how memory-based Lyapunov functions and their path-complete counterparts can be beneficial for the analysis of dynamical systems. 
We consider a linear switched system already studied in~\cite[Example 5.4]{AJPR:14}, and show how a particular (and non-standard) memory structure can provide a numerically appealing stability criterion with respect to more classical approaches, without losing in conservatism.
\begin{example}
We consider the alphabet $\cS=\{a,b\}$, and we consider the \emph{linear} switched system
\begin{equation}\label{eq:FirstExample}
x(k+1)=A_{\sigma(k)}x(k)
\end{equation}
where $\sigma:\N\to \cS$, and
\[
A_a:=\begin{bmatrix}3 & \,3\\ -2 & \,1\end{bmatrix},\;\;\; A_b:=\begin{bmatrix}-1 & -1 \\ -4 & 0  \end{bmatrix}.
\]
We want to use \emph{quadratic} memory-based Lyapunov functions (as in Definition~\ref{defn:FiniteMBLF}) to estimate the maximal growth rate of solutions of~\eqref{eq:FirstExample}, a.k.a. the \emph{joint spectral radius} of $\{A_a,A_b\}$ denoted by $\rho(A_a,A_b)$, see~\cite{Jung09} for more discussion. More formally, we consider the set of quadratic functions $\cQ=\{f(x):=x^\top P x\;\;\vert\;P\in \R^{2\times 2}, P\succ 0\}$. Given a covering family $\cC$ on $\cS$, we consider the corresponding graph $\cG_\cC=(N_\cC,E_\cC)$ given by Theorem~\ref{thm:FirstTransl}. Then, the best upper bound of $\rho(A_a,A_b)$ given by quadratic $\cC$-based Lyapunov functions, denoted by $\rho_{\cC,\cQ}(A_a,A_b\}$, is given by the solution of the following semi-definite optimization problem:
\begin{equation}\label{eq:optimprob}
\begin{aligned}
&\min_{\rho>0,P_s\in \R^{2\times 2}} \rho,\;\;\;\text{ s.t.}
\\
&P_s\succ 0,\;\;\forall s\in N_\cC\\
&A_h^\top P_qA_h-\rho^2P_r\prec 0,\;\;\forall (r,q,h)\in E_\cC,
\end{aligned}
\end{equation}
i.e. we search quadratic functions satisfying the conditions in Definition~\ref{defn:FiniteMBLF}, minimizing the parameter $\rho$. We compare three different coverings: $\cC_1=\{[a],[b]\}$ made of all the prefix classes of length $1$ (with its graph structure depicted in Figure~3); $\cC_2=\{[aa],[ab],[ba],[bb]\}$ i.e. the partition of all the prefix classes of length $2$ (with its graph structure depicted in Figure~4), and the partition $\cD=\{[aa],[ab],[b]\}$, already considered in Example~\ref{example:prefixclasss2}, and represented in Figure~\ref{figure:Figure1}. We underline that the conditions arising from the coverings $\cC_1$ and $\cC_2$, since they are composed by fixed-length prefix classes, correspond to particular instances of the conditions proposed in~\cite{EssickLee14,Thesis:Essick}, as already underlined in Remark~\ref{rem:Cyliders}.
One can see that the memory information contained in $\cC_1$ is lesser than the one contained in $\cD$, which is lesser than the one in $\cC_2$. As a result, one can show that
\[
\rho_{\cC_2,\cQ}(A_a,A_b)\leq \rho_{\cD,\cQ}(A_a,A_b)\leq\rho_{\cC_1,\cQ}(A_a,A_b)
\]
i.e. the estimation of the JSR of $\{A_a,A_b\}$ given by the memory structure $\cC_2$ is not worst than the one provided by $\cD$, which in turn is not worst than the one given by $\cC_1$. We leave the formalization of this result for further work because of space constraints. In TABLE~I  we report the values obtained by solving the optimization problem, using \textsc{SeDuMi} solver in \textsc{Matlab}, see~\cite{Yalmip04}.
\begin{table}[b!]
\vspace{0.1cm}
\centering
\caption{\normalfont {Numerical upper bounds of maximal growth rate of~\eqref{eq:FirstExample}, obtained with different finite-memory structures.}}
\begin{tabular}{ |c|c|c|c|} 
 \hline
Mem. Structure : & $\cC_1$ & $\cD$ & $\cC_2$  \\ 
 \hline
  $\rho_{\cC,\cQ}(\cA,\cG)$ & $3.9224$ & $3.9174$ & $3.9174$ \\
  \hline
\end{tabular}
\end{table}

The results in the table show the following: memory-based conditions arising from $\cD$ and $\cC_2$ are \emph{strictly better} than the ones induced by $\cC_1$. On the other hand, there is no remarkable difference between the upper bound provided by $\cD$ and $\cC_2$. In other words, in this particular case, the fact of storing strings of length $2$ is not beneficial with respect to conditions requiring to store strings of length $2$ \emph{only if} the previous active system is $a$ (i.e. the conditions encoded in $\cD$). We stress that, from a numerical point of view, the conditions induced by $\cD$ are preferable, since the corresponding semidefinite optimization problem~\eqref{eq:optimprob} (for a fixed $\rho>0$) has $3$ semidefinite variables and $6$ LMIs (as constraints) while the optimization problem corresponding to $\cC_2$ has $4$ semidefinite variables and $8$ constraints inequality.
\end{example}

This simple example showed that, even in very restricted settings ($2$-mode planar linear switched systems), considering non-uniform memories (as it was the case for the partition given by $\cD$) can improve, from a numerical point of view, the stability conditions arising when considering conditions based on fixed-length memories as in~\cite{EssickLee14,Thesis:Essick}.

\begin{figure}[t!] \label{Fig:FirstDeBrujin}
  \centering
\begin{tikzpicture}%
  [>=stealth,
  shorten >=0.8pt,
  node distance=3cm,
  on grid,
  auto,
  every state/.style={draw=black, fill=white, thick}
  ]
\node[state,inner sep=1pt, minimum size=22pt] (left)                  {$[a]$};
\node[state,inner sep=1pt, minimum size=22pt] (right) [right=of left] {$[b]$};
\path[->]
   (left) edge[bend left=20]     node          [scale=1]            {$b$} (right)
        (right)   edge[bend left=20] node        [scale=1]        {$a$} (left)
   (left) edge[loop left=60]     node            [scale=1]          {$a$} (left)
   (right) edge[loop right=60]     node      [scale=1]                {$b$} (right)
   ;
\end{tikzpicture}
\label{figure:DeBrujin1}
\caption{The graph $\cG_1$, corresponding to the covering $\cC_1=\{[a],[b]\}$, a.k.a. the De Bruijn graph of order $1$.}
  \end{figure}
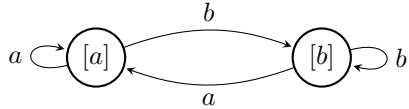

  \begin{figure} \label{Fig:2DeBrujin}
  \centering
\begin{tikzpicture}%
  [>=stealth,
  shorten >=1pt,
  node distance=1.2cm,
  on grid,
  auto,
  every state/.style={draw=black, fill=white, thick}
  ]
  \node[state, inner sep=1pt, minimum size=23pt] (left)                  {$[aa]$};
  \node[state, inner sep=1pt, minimum size=23pt] (right) [right=of left, xshift=3.3cm] {$[bb]$};
  \node[state, inner sep=1pt, minimum size=23pt] (upper) [above right=of left, xshift=1.5cm, yshift=-0.05cm]{$[ba]$};
  \node[state, inner sep=1pt, minimum size=23pt] (below) [below right=of left, xshift=1.5cm, yshift=0.05cm]{$[ab]$};
  \path[->]
  (left) edge[loop left=60]     node             [scale=0.9]             {$a$} (left)
  (left) edge[bend left=15]     node                      {$b$} (upper)
  (upper) edge[bend left=15]     node                      {$a$} (below)
  (below) edge[bend left=15]     node                      {$a$} (left)
  (below) edge[bend left=15]     node                      {$b$} (upper)
  (right) edge[bend left=15]     node                      {$a$} (below)
  (upper) edge[bend left=15]     node                      {$b$} (right)
  (right) edge[loop right=300]     node            [scale=0.9]              {$b$} (right)
;
  \end{tikzpicture}

  \caption{The graph $\cG_2$, corresponding to the covering $\cC_2:=\{[aa],[ab],[ba],[bb]\}$, a.k.a. the De Bruijn graph of order $2$.}
\end{figure}
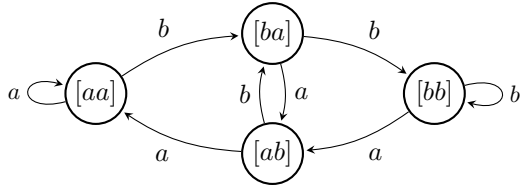

\section{Conclusion}\label{sec:Conclu}
In this work, we presented a new interpretation of path-complete Lyapunov functions, in terms of time-dependent Lyapunov function structure, whose dependence in time is restricted to the last values taken by the switching signal (the ``memory'').
Our result provides interpretability for the path-complete Lyapunov techniques, and this may be useful for control engineers, in that it relates stability properties with the nature of the observation.
From a more quantitative point of view, our results allow to compress Lyapunov criteria, as seen in Section~ \ref{sec:NumExample}.
In the future, we will leverage this new interpretation of Path-Complete Lyapunov functions to handle more general systems for which information about the state can be compressed in a discrete set of observations.
We will also leverage our results in order to formalize performance order relations between different path-complete Lyapunov functions. 

\bibliography{biblio} 
\bibliographystyle{plain}
\end{document}